\documentclass[12pt]{article}

\usepackage{latexsym,amsfonts,amsmath,amssymb,mathrsfs,amsthm}

\newtheorem{theorem}{Theorem}
\newtheorem{lemma}[theorem]{Lemma}

\newtheorem{remark}[theorem]{Remark}
\newtheorem{proposition}[theorem]{Proposition}
\newtheorem{definition}[theorem]{Definition}

\newcommand{\icomp}{\mathtt{i}}
\newcommand{\rd}{\,\mathrm{d}}

\newcommand{\bsx}{\boldsymbol{x}}
\newcommand{\bsy}{\boldsymbol{y}}

\newcommand{\bsl}{\boldsymbol{l}}
\newcommand{\bsk}{\boldsymbol{k}}
\newcommand{\bsh}{\boldsymbol{h}}
\newcommand{\bsell}{\boldsymbol{l}}
\newcommand{\bst}{\boldsymbol{t}}

\newcommand{\bssigma}{\boldsymbol{\sigma}}
\newcommand{\bszero}{\boldsymbol{0}}

\newcommand{\uu}{\mathfrak{u}}

\newcommand{\FF}{\ZZ}
\newcommand{\NN}{\mathbb{N}}

\newcommand{\Dcal}{\mathcal{D}}
\newcommand{\Ecal}{\mathcal{E}}
\newcommand{\Jcal}{\mathcal{J}}
\newcommand{\ZZ}{\mathbb{Z}}
\newcommand{\Mod}{\operatorname{mod}}
\newcommand{\mypmod}[1]{\,(\Mod\,#1)}
\newcommand{\wal}{{\rm wal}}
\newcommand{\cP}{\mathcal{P}}
\newcommand{\cS}{\mathcal{S}}

\newcommand{\satop}[2]{\stackrel{\scriptstyle{#1}}{\scriptstyle{#2}}}

\allowdisplaybreaks

\begin{document}

\title{Optimal periodic $L_2$-discrepancy and diaphony bounds for higher order digital sequences}

\author{Friedrich Pillichshammer\thanks{ORCID ID 0000-0001-6952-9218. The author is supported by the Austrian Research Foundation (FWF), Project F5509-N26, that is part of the Special Research Program ``Quasi-Monte Carlo Methods: Theory and Applications''.}}

\date{}

\maketitle

\begin{abstract}
We present an explicit construction of infinite sequences of points $(\bsx_0,\bsx_1, \bsx_2, \ldots)$ in the $d$-dimensional unit-cube whose periodic $L_2$-discrepancy satisfies $$L_{2,N}^{{\rm per}}(\{\bsx_0,\bsx_1,\ldots, \bsx_{N-1}\}) \le C_d  N^{-1} (\log N)^{d/2} \quad \mbox{for all } N \ge 2,$$ where the factor $C_d > 0$ depends only on the dimension $d$. The construction is based on higher order digital sequences as introduced by J. Dick in the year 2008. The result is best possible in the order of magnitude in $N$ according to a Roth-type lower bound shown first by P.D. Proinov. Since the periodic $L_2$-discrepancy is equivalent to P. Zinterhof's diaphony the result also applies to this alternative quantitative measure for the irregularity of distribution. 
\end{abstract}



\section{Introduction and Statement of the Main Result}

We study distribution properties of point sequences in the $d$-dimensional unit-cube $[0,1)^d$ by means of periodic $L_2$-discrepancy and of diaphony.

\paragraph{Periodic $L_2$-discrepancy.} The periodic $L_2$-discrepancy is the $L_2$ mean of the local discrepancy function with respect to test sets from the class of ``periodic intervals''. This is in contrast to the usual $L_2$-discrepancy which uses as test sets exclusively subintervals of the unit-cube that are anchored in the origin, i.e., intervals of the form $[\bszero,\bst):=[0,t_1) \times \ldots \times [0,t_d)$, where $\bst=(t_1,\ldots,t_d)$ in $[0,1]^d$. See \cite{BC, DP10, DT97,kuinie,mat} for general information about (standard) $L_2$-discrepancy. 

For $x,y\in [0,1]$ set 
$$ I(x,y)=\begin{cases}
           [x,y) & \text{if $x\leq y$}, \\
           [0,y)\cup [x,1)& \text{if $x>y$,}
          \end{cases}$$
and for $\bsx,\bsy \in [0,1]^d$ with $\bsx=(x_1,\ldots,x_d)$ and $\bsy=(y_1,\ldots,y_d)$ we set $B(\bsx,\bsy)=I(x_1,y_1)\times \dots \times I(x_d,y_d)$. We call $B(\bsx,\bsy)$ a ``periodic interval''. 

For a finite set $\cP_{N,d} =\{\bsx_0,\bsx_1,\ldots ,\bsx_{N-1}\}$ of points in the $d$-dimensional unit-cube $[0,1)^d$ the periodic $L_2$-discrepancy is defined as
$$ L_{2,N}^{\mathrm{per}}(\cP):=\left(\int_{[0,1]^d}\int_{[0,1]^d}\left|\frac{A(B(\bsx,\bsy),\cP_{N,d})}{N}-\lambda_d(B(\bsx,\bsy))\right|^2\rd \bsx\rd\bsy\right)^{1/2},$$ where the counting function $A(B(\bsx,\bsy),\cP_{N,d})$  denotes the number of indices $n$ with $\bsx_n \in B(\bsx,\bsy)$ and where $\lambda_d(B(\bsx,\bsy))$ denotes the volume of $B(\bsx,\bsy)$. The term $A(B(\bsx,\bsy),\cP_{N,d})/N-\lambda_d(B(\bsx,\bsy))$ is sometimes referred to as the ``local discrepancy'' of $\cP_{N,d}$ for the test set $B(\bsx,\bsy)$. The local discrepancy measures the difference of the portion of points in a periodic interval and the volume of this periodic interval. Hence it is a measure for the irregularity of distribution of a point set in $[0,1)^d$.

For an infinite sequence $\cS_d=(\bsx_0,\bsx_1,\bsx_2,\ldots)$ in $[0,1)^d$ the periodic $L_2$-discrepancy $L_{2,N}^{{\rm per}}(\cS_d)$ is the periodic $L_2$ discrepancy of the point set consisting of the first $N$ elements of $\cS_d$, i.e., $$L_{2,N}^{{\rm per}}(\cS_d):=L_{2,N}^{{\rm per}}(\{\bsx_0,\bsx_1,\ldots,\bsx_{N-1}\}) \quad \mbox{for $N \in \NN$.}$$

The periodic $L_2$-discrepancy has been studied recently in several papers. See, for example, \cite{DHP20,HKP20,HOe,HiWe,KP22a,KP22b,Lev}.

\paragraph{Diaphony.}  The diaphony, as introduced by Zinterhof~\cite{zint} in the year 1976 (see also \cite{DT97}), is just another quantitative measure for the irregularity of distribution of point sets and sequences in the unit-cube $[0,1)^d$. For a finite set $\cP_{N,d} =\{\bsx_0,\bsx_1,\ldots ,\bsx_{N-1}\}$ like above the diaphony is defined as $$F_N(\cP_{N,d}):=\left(\sum_{\bsh \in \ZZ^d\setminus\{\bszero\}} \frac{1}{\rho(\bsh)^2} \left|\frac{1}{N} \sum_{n=0}^{N-1} \exp(2 \pi \icomp \bsh \cdot \bsx_n)\right|^2\right)^{1/2},$$ where $\icomp=\sqrt{-1}$, where $\rho(\bsh):=\prod_{j=1}^d \max(1,|h_j|)$ for $\bsh=(h_1,\ldots,h_d) \in \ZZ^d$ and where ``$\cdot$'' denotes the usual Euclidean inner product. For an infinite sequence $\cS_d$ the diaphony $F_N(\cS_d)$ is the diaphony of the initial $N$ elements of $\cS_d$, i.e., $$F_N(\cS_d):=F_N(\{\bsx_0,\bsx_1,\ldots,\bsx_{N-1}\})\quad \mbox{for $N \in \NN$.}$$ 

The diaphony is a popular quantitative measure for the irregularity of distribution in the sense that a sequence $\cS_d$ is uniformly distributed modulo one (in the sense of Weyl~\cite{weyl}), if and only if $\lim_{N \rightarrow \infty} F_N(\cS_d)=0$; see, for example, \cite[Theorem~1.33]{DT97}. Diaphony of point sets and sequences is studied in a multitude of papers, see, for example,  \cite{chafa,F2005,g96,Lev0,pag,PS,pro1988a,pg} and the references therein. Relations between the number-theoretic concept of diaphony and crystallographic concepts are discussed in \cite{HK}.

\paragraph{Relations and known results.} It is a known fact that the periodic $L_2$-discrepancy can be expressed in terms of exponential sums. For $\cP_{N,d}=\{\bsx_0,\bsx_1,\ldots,\bsx_{N-1}\}$ in $[0,1)^d$ we have 
\begin{equation}\label{pr_dia}
L_{2,N}^{{\rm per}}(\cP_{N,d})=\left(\frac{1}{3^d} \sum_{\bsh \in \ZZ^d\setminus\{\bszero\}} \frac{1}{r(\bsh)^2} \left|\frac{1}{N} \sum_{n=0}^{N-1} \exp(2 \pi \icomp \bsh \cdot \bsx_n)\right|^2\right)^{1/2},
\end{equation} 
where for $\bsh=(h_1,\ldots,h_d)\in \ZZ^d$ we set 
\begin{equation*}
r(\bsh)=\prod_{j=1}^d r(h_j) \ \ \ \mbox{ and } \ \ r(h_j)=\left\{ 
\begin{array}{ll}
1 & \mbox{ if $h_j=0$},\\
\frac{2 \pi |h_j|}{\sqrt{6}} & \mbox{ if $h_j\not=0$.} 
\end{array}\right.
\end{equation*}
For a proof of this relation see \cite[Theorem~1]{Lev} or \cite[p.~390]{HOe}. Formula \eqref{pr_dia} shows that the periodic $L_2$-discrepancy is -- up to a multiplicative factor depending on $d$ -- exactly the diaphony. 

We briefly explain some notation: For quantities $A(N,d)$ and $B(N,d)$ depending on $N$ and $d$ we write $A(N,d) \lesssim_d B(N,d)$ if there is a factor $c_d > 0$ which depends only on $d$ (and not on $N$) such that $A(N,d) \le c_d B(N,d)$. If we have $A(N,d) \lesssim_d B(N,d)$ and $B(N,d) \lesssim_d A(N,d)$ simultaneously, then we write $A(N,d) \asymp_d B(N,d)$. 

Using this notation we can express the fact that periodic $L_2$-discrepancy and diaphony are the same up to a factor depending only on $d$ in the form $$L_{2,N}^{{\rm per}}(\cP_{N,d})  \asymp_d F_N(\cP_{N,d}).$$ So the periodic $L_2$-discrepancy can be understood as a geometrical interpretation of the diaphony, which is of analytic nature. From this point of view, the results about periodic $L_2$-discrepancy presented in this paper apply directly also to the diaphony. Both measures can be also interpreted as worst-case errors of quasi-Monte Carlo integration rules in suitable function classes. For discussions in this direction we refer to \cite{HKP20,HOe,pag}.

In classical discrepancy theory (see \cite{BC,DT97,kuinie,mat}) one is interested in how small the discrepancy can be in the best case and in constructions of point sets and sequences whose discrepancy is of the optimal order of magnitude. For the periodic $L_2$-discrepancy and diaphony, respectively, the following is known:

\paragraph{Finite point sets.} For every dimension $d \in \NN$ there exists a quantity $c_d>0$ which only depends on $d$, such that for every $N \in \NN$ and every $N$-element point set $\cP_{N,d}$  in $[0,1)^d$ we have 
\begin{equation}\label{LB:pts}
L_{2,N}^{{\rm per}}(\cP_{N,d}) \ge c_d \frac{(\log N)^{(d-1)/2}}{N}.
\end{equation} 
See, for example, \cite[Corollary~2]{HKP20}. This is the ``periodic'' counterpart of Roth's famous lower bound for the usual $L_2$-discrepancy from \cite{Roth}. The lower bound \eqref{LB:pts} is known to be best possible in dimension $d=2$, see, for example, \cite{g99,HKP20}. For arbitrary dimension $d$, it follows from \cite[Theorem~5.3]{HMOU} (see the discussion in \cite[Remark~18]{KP22b}) that the periodic $L_2$-discrepancy of order 2 digital $(t,m,d)$-nets over $\ZZ_2$ is of order of magnitude $(\log N)^{(d-1)/2}/N$, where $N=2^m$, which matches the lower bound \eqref{LB:pts}. Another ``semi'' construction, based on digitally shifted digital nets in prime base $b$ is presented in \cite[Section~4]{KP22b} (but unfortunately here the digital shifts leading to the optimal order of magnitude are unknown, which is indicated by the term ``semi construction''). 

\begin{remark}\label{re:lev}\rm 
Just to avoid misunderstandings we mention that Lev~\cite{Lev0} studied a slightly more general notion of diaphony which involves certain weights and which he called generalized diaphony. For this kind of generalized diaphony he showed the existence of certain nets (where Lev means point sets together with suitable positive weights for every point) yielding optimal order of magnitude (see \cite[Main Theorem]{Lev0}). However, these upper bounds are only achieved for certain weights which do not comprise the classical setting considered in the present paper. 
\end{remark}

\paragraph{Infinite sequences.} For every dimension $d \in \NN$ there exists a quantity $c'_d>0$ which only depends on $d$, such that for every infinite sequence $\cS_d$ of points in $[0,1)^d$ we have 
\begin{equation}\label{LB:seq}
L_{2,N}^{{\rm per}}(\cS_d) \ge c'_d \frac{(\log N)^{d/2}}{N}\quad \mbox{for infinitely many $N \in \NN$.}
\end{equation}
This has been shown first in the context of diaphony by Pro{\u\i}nov~\cite{pro2000} (this work is only available in Bulgarian language). A discussion of Pro{\u\i}nov's work can be found in \cite{kirk}. Later, a direct proof in terms of periodic $L_2$-discrepancy has been provided in \cite{KP22a}. The lower bound \eqref{LB:seq} is known to be best possible in dimension $d=1$. One-dimensional infinite sequences whose periodic $L_2$-discrepancy satisfies a bound of order of magnitude $\sqrt{\log N}/N$ for every $N \ge 2$ were given in, e.g., \cite{chafa,g96,pag,pro1988a,pg}. These constructions comprise the fundamental van der Corput sequence as shown by Pro{\u\i}nov and Grozdanov~\cite{pg} (for information about the van der Corput sequence see \cite{FKP}). On the other hand, although it suggests itself to believe that the lower bound \eqref{LB:seq} is also best possible for arbitrary dimensions $d \ge 2$, so far there was no proof for this assertion.  

\paragraph{The main result.} In this paper we give for the first time an explicit construction of {\it infinite} sequences in $[0,1)^d$ for arbitrary dimension $d$, whose periodic $L_2$-discrepancy is of order of magnitude $(\log N)^{d/2}/N$ in $N$ which is best possible according to \eqref{LB:seq}. In fact, our result even gives a slightly more refined description of the whole situation.  Our main result is the following:

\begin{theorem}\label{thm1}
For every $d \in \NN$ one can explicitly construct an infinite sequence $\cS_d$ of points in $[0,1)^d$ such that for all $N \ge 2$ we have
\begin{equation*}
L_{2,N}^{{\rm per}}(\cS_d) \lesssim_d  \frac{(\log N)^{(d-1)/2}}{N} \sqrt{S(N)} \lesssim_d \frac{(\log N)^{d/2}}{N},
\end{equation*}
where $S(N)$ is the sum-of-digits function of $N$ in base 2 representation, i.e., if $N = 2^{n_1} + 2^{n_2} + \cdots + 2^{n_r}$ with $n_1 > n_2 > \cdots > n_r \ge 0$, then $S(N)=r$. 
\end{theorem}

\begin{remark}\rm
Obviously, we have $S(N) \lesssim \log N$ for all $N \in \NN$ and $S(N)=(\log(N+1))/\log 2$ if $N$ is of the form $N=2^m-1$ with $m \in \NN$. However, for certain $N$ the sum-of-digits function can be much smaller. For example, if $N$ is a power of 2, then $S(N)=1$ only. Hence the main theorem describes an infinite class of integers $N \in \NN$, for which the periodic $L_2$-discrepancy of $\cS_d$ is much smaller than the worse lower bound from \eqref{LB:seq} which holds for infinitely many $N \in \NN$. 
\end{remark}

The explicit construction of infinite sequences from Theorem~\ref{thm1} will be presented in Section~\ref{sec_construction}. Theorem~\ref{thm1} will then be a consequence of Theorem~\ref{thm2} that will be stated at the end of Section~\ref{sec_construction}.

\section{Explicit Constructions of Sequences}\label{sec_construction}

We now present explicit constructions of sequences whose periodic $L_2$-discrepancy satisfies the bound from Theorem~\ref{thm1}. We use the same construction like in \cite{DP14} for sequences with optimal order of the standard $L_2$-discrepancy. The general construction principle was introduced by Dick in \cite{D07,D08} and is based on linear algebra over the finite field $\ZZ_2$ of order $2$ (we identify $\ZZ_2$ with the set $\{0,1\}$ equipped with the arithmetic operations modulo 2). Dick's construction in turn is based on digital sequences according to Niederreiter \cite{nie87,nie92}. 

\paragraph{Digital sequences.}  Basic module for the construction of a digital sequence in $[0,1)^d$ are $d$ (one for each coordinate direction) $\NN \times \NN$ matrices $C_1,\ldots, C_d$ with entries from $\ZZ_2$. In order that the points of a digital sequence belong to $[0,1)^d$ it is commonly assumed that for $C_j=(c_{j,k,l})_{k,l \in \NN}$ for each $l \in \NN$ the elements of column $l$ become zero eventually, i.e., there exists a number $K(l) \in \NN$, such that $c_{j,k,l}=0$ for all $k> K(l)$.

For $n \in \NN_0$, where $\NN_0:=\NN \cup \{0\}$, the $n$-th point of the digital sequence is constructed in the following way. Consider the binary digit expansion of $n$ which is of the form $n = n_0 + n_1 2 + \cdots + n_{m-1} 2^{m-1}$ with $m \in \NN$ and $n_i \in \{0,1\}$ for every $i \in \{0,1,\ldots,m-1\}$,  and define the binary digit vector of $n$ as $\vec{n} := (n_0, n_1, \ldots, n_{m-1}, 0, 0, \ldots )^\top \in \ZZ_2^{\mathbb{N}}$. Then, for $j \in \{1,\ldots, d\}$, define
\begin{equation*}
\vec{x}_{j,n}= (x_{j,n,1}, x_{j,n,2},\ldots)^\top := C_j \vec{n},
\end{equation*}
where the matrix vector product is evaluated over $\ZZ_2$, and 
\begin{equation*}
x_{j,n} := \frac{x_{j,n,1}}{2} + \frac{x_{j,n,2}}{2^2} +  \frac{x_{j,n,3}}{2^3} + \cdots.
\end{equation*}
Then the $n$-th point $\bsx_n$ of the sequence $\cS_d$ is given by $\bsx_n = (x_{1,n}, \ldots, x_{d,n})$. A sequence $\cS_d$ constructed this way is called a {\it digital sequence} (over $\ZZ_2$) with {\it generating matrices} $C_1,\ldots,C_d$.

\begin{remark}\label{re:dnet}\rm
A finite version of digital sequences are so-called {\it digital nets}, which consist of $2^m$ elements in $[0,1)^d$ and which are constructed in exactly the same way like digital sequence with the restriction that $C_1,\ldots,C_d$ are (finite) $p \times m$ matrices over $\ZZ_2$ with fixed $m, p \in \NN$, $p \ge m$, and $n \in \{0,1,\ldots,2^m-1\}$. For more information see \cite{DP10,nie87,nie92}.
\end{remark}

Explicit constructions of suitable generating matrices $C_1,\ldots, C_d$ for digital sequences were obtained by Sobol'~\cite{sob67}, Niederreiter~\cite{nie87,nie92}, Niederreiter-Xing~\cite{NX96} and others (see also \cite[Chapter~8]{DP10}). 

\paragraph{A concrete construction.} We briefly describe a special case of Sobol's and Nieder\-reiter's construction of the generating matrices. Let $C_j = (c_{j,k,l})_{k,l \ge 1}$ with $c_{j,k,l} \in \ZZ_2$. Let $p_1=x$ and $p_j \in \ZZ_2[x]$ for $j \in \{2,\ldots,d-1\}$ be the $(j-1)$-st primitive polynomial in a list of primitive polynomials over $\ZZ_2$ that is sorted in increasing order according to their degree $e_j = \deg(p_j)$, that is, $e_2 \le e_3 \le \cdots \le e_{d-1}$ (the ordering of polynomials with the same degree is irrelevant). We also put $e_1=\deg(x)=1$.

Let $j \in \{1,\ldots,d\}$ and $k \in \NN$. Take $i-1$ and $z$ to be respectively the main term and remainder when we divide $k-1$ by $e_j$, so that   $k-1  = (i-1) e_j + z$, with $z \in \{0,1,\ldots,e_j-1\}$. Now consider the Laurent series expansion
\begin{equation*}
\frac{x^{e_j-z-1}}{p_j(x)^i} = \sum_{l =1}^\infty \frac{a_l(i,j,z)}{x^{l}} \in \ZZ_2((x^{-1}))
\end{equation*}
and  write, for $l \in \NN$,
\begin{equation}\label{def:gNs}
c_{j,k,l} = a_l(i,j,z).
\end{equation}

Now a digital sequence with generating matrices $C_j=(c_{j,k,l})_{k,l \ge 1}$ constructed this way is a special instance of a Sobol' sequence (which itself is a special instance of a generalized Niederreiter sequence). For more information on these constructions we refer to \cite[Chapter~8]{DP10}. Observe that for the presented construction of generating matrices we have $c_{j,k,l}=0$ for all $k>l$.

\paragraph{The general construction principle.} The present construction of sequences with the optimal order of periodic $L_2$-discrepancy is based on higher order digital sequences as introduced by Dick in \cite{D07,D08}. We state here simplified versions of their definitions that are sufficient for the present purpose. 

The distribution quality of digital sequences depends on the choice of the respective generating matrices. In the following definitions we put some restrictions on $C_1,\ldots ,C_d$ with the aim to quantify the quality of equidistribution of the digital sequence. 

First we give the definition for the finite case as introduced in Remark~\ref{re:dnet}.

\begin{definition}\label{def_net}\rm
Let $m, p, \alpha \in \NN$ with $p \ge \alpha m$ and let $t$ be an integer such that $0 \le t \le \alpha m$. Let $C_1,\ldots, C_d$ be $p \times m$ matrices with entries from $\ZZ_2$. For $j \in \{1,\ldots,d\}$ and $i \in \{1,\ldots,p\}$ let $\vec{c}_{j,i} \in \ZZ_2^m$ be the $i$-th row vector of the matrix $C_j$. If for all $1 \le i_{j,\nu_j} < \cdots <
i_{j,1} \le p$ with $$\sum_{j = 1}^d \sum_{l=1}^{\min(\nu_j,\alpha)} i_{j,l}  \le
\alpha m - t$$ the vectors
$$\vec{c}_{1,i_{1,\nu_1}}, \ldots, \vec{c}_{1,i_{1,1}}, \ldots,
\vec{c}_{d,i_{d,\nu_d}}, \ldots, \vec{c}_{d,i_{d,1}}$$ are linearly independent
over $\ZZ_2$, then the digital net constructed from the generating matrices $C_1,\ldots,C_d$ is called an {\it order $\alpha$ digital $(t,m,d)$-net over $\ZZ_2$}. We then say that the matrices $C_1,\ldots,C_d$ satisfy the {\it order $\alpha$ digital $(t,m,d)$-net property}.
\end{definition}

\begin{definition}\rm\label{def_seq}
Let $\alpha \in \NN$ and let $t \ge 0$ be an integer. Let $C_1,\ldots, C_d \in \ZZ_2^{\mathbb{N} \times \mathbb{N}}$ and let $C_{j, \alpha m \times m}$ denote the left upper $\alpha m \times m$ submatrix of  $C_j$ for $j \in \{1,\ldots,d\}$. If for all $m > t/\alpha$ the matrices $C_{1, \alpha m \times m},\ldots, C_{d, \alpha m \times m}$ satisfy the order $\alpha$ digital $(t, m,d)$-net property, then the digital sequence with generating matrices $C_1,\ldots, C_d$ is called an {\it order $\alpha$ digital $(t,d)$-sequence over $\FF_2$}.
\end{definition}

From Definition~\ref{def_seq} it is clear that if $\cS_d$ is an order $\alpha$ digital $(t,d)$-sequence, then for any $t \le t'$, it is also an order $\alpha$ digital $(t',d)$-sequence. If $\alpha =1$, then the concept of order 1 digital $(t,d)$-sequences is exactly the same like digital $(t,d)$-sequences as introduced by Niederreiter in~\cite{nie87,nie92}. For a geometrical interpretation of the order $\alpha$ digital net property we refer to \cite[Chapter~15]{DP10} or to \cite[Section~1.4]{DP14}.

Note that a digital sequence can be an order $\alpha$ digital $(t,d)$-sequence over $\ZZ_2$ and at the same time an order $\alpha'$ digital $(t',d)$-net over $\ZZ_2$ for $\alpha'\not = \alpha$. This means that the quality parameter $t$ may depend on $\alpha$. If necessary we write $t(\alpha)$ instead of $t$ for the quality parameter of an order $\alpha$ digital $(t(\alpha),d)$-sequence. In particular \cite[Theorem~4.10]{D08} implies that an order $\alpha$ digital $(t,d)$-sequence is an order $\alpha'$ digital $(t',d)$-sequence for all $1 \le \alpha' \le \alpha$ with
\begin{equation}\label{eq_t_tprime}
t' = \lceil t \alpha'/\alpha \rceil \le t.
\end{equation}
In other words, $t(\alpha') = \lceil t(\alpha) \alpha'/\alpha \rceil$ for all $1 \le \alpha' \le \alpha$. More information can be found in \cite[Chapter~15]{DP10}. 

Next we present a concrete construction of an order $\alpha$ digital $(t,d)$-sequence.

\paragraph{A concrete construction of order $\alpha$ digital sequences.} In order to construct order $\alpha$ digital sequences we need the following composition principle.

\begin{definition}\rm
For $\alpha \in \NN$ the {\it digit interlacing composition} (with interlacing factor $\alpha$) is defined by
\begin{eqnarray*}
\mathscr{D}_\alpha: [0,1)^{\alpha} & \to & [0,1) \\
(x_1,\ldots, x_{\alpha}) &\mapsto & \sum_{a=1}^\infty \sum_{r=1}^\alpha
\xi_{r,a} 2^{-r - (a-1) \alpha},
\end{eqnarray*}
where $x_r \in [0,1)$ has binary digit expansion of the form $x_r = \xi_{r,1} 2^{-1} + \xi_{r,2} 2^{-2} + \cdots$ with digits $\xi_{r,j} \in \{0,1\}$ for $j \in \NN$ and $r \in \{1,\ldots,\alpha\}$. We also define this function for vectors by setting
\begin{eqnarray*}
\mathscr{D}_\alpha^d : [0,1)^{\alpha d} & \to & [0,1)^d \\
(x_1,\ldots, x_{\alpha d}) &\mapsto & (\mathscr{D}_\alpha(x_1,\ldots, x_\alpha),  \ldots, \mathscr{D}_\alpha(x_{(d-1) \alpha +1},\ldots, x_{\alpha d}))
\end{eqnarray*}
and for sequences $\cS_{\alpha d} = (\bsx_0, \bsx_1, \ldots)$ with $\bsx_n \in [0,1)^{\alpha d}$ by setting
\begin{equation*}
\mathscr{D}_\alpha^d(\cS_{\alpha d}) = (\mathscr{D}_\alpha^d(\bsx_0), \mathscr{D}_\alpha^d(\bsx_1), \ldots).
\end{equation*}
\end{definition}

Likewise, the interlacing can also be applied to the generating matrices $C_1,\ldots, C_{\alpha d}$ directly. This is described in \cite[Section~4.4]{D08} and can be done in the following way. Let $C_1, \ldots, C_{\alpha d}$ be generating matrices of an  $\alpha d$-dimensional digital sequence and let $\vec{c}_{j,k}$ denote the $k$-th row of matrix $C_j$. Define matrices $E_1,\ldots, E_d$, where the $k$-th row of matrix $E_j$ is given by $\vec{e}_{j,k}$, in the following way. For all $j \in \{1,\ldots,d\}$, $u \in \NN_0$ and $v \in \{1,\ldots,\alpha\}$ let
\begin{equation*}
\vec{e}_{j,u\alpha + v} := \vec{c}_{(j-1) \alpha + v, u+1}
\end{equation*}
If $C_1, \ldots, C_{\alpha d}$ are the generating matrices of a digital sequence $\cS_{\alpha d}$ in dimension $\alpha d$, then the matrices $E_1,\ldots, E_d$ defined above, are the generating matrices of $\mathscr{D}_\alpha^d(\cS_{\alpha d})$. Thus one can also obtain generating matrices $E_1,\ldots, E_d \in \ZZ_2^{\mathbb{N} \times \mathbb{N}}$ which generate a digital sequence satisfying Theorem~\ref{thm1}.

The following result follows from \cite[Theorem~4.11 and Theorem~4.12]{D07} (where we set $\alpha = d$).

\begin{proposition}
If $\cS_{\alpha d}$ is an order 1 digital $(t',\alpha d)$-sequence over $\ZZ_2$, then $\mathscr{D}_\alpha^d(\cS_{\alpha d})$ is an order $\alpha$ digital $(t,d)$-sequence over $\ZZ_2$ with $$t = \alpha t' + d {\alpha \choose 2}.$$
\end{proposition}

For the construction based on Sobol's and Niederreiter's sequence introduced above we have $$t=\sum_{j=1}^d (e_j-1)$$ (see  \cite[Section~8.1]{DP10} for details) and therefore we obtain explicit constructions of order $\alpha$ digital $(t,d)$-sequences over $\ZZ_2$ with
\begin{equation*}
t = \alpha \sum_{j=1}^{\alpha d} (e_j-1) + d {\alpha \choose 2}.
\end{equation*}

Note that for the construction introduced above we have $c_{j,k, l} = 0$ for all $k > l$. Using the interlacing construction we obtain generating matrices $E_1,\ldots, E_d$ with $E_j = (e_{j,k,l})_{k,l \in \mathbb{N}}$ and 
\begin{equation}\label{matE:null}
e_{j,k,l} = 0\quad \mbox{for all $k > \alpha l$.}
\end{equation} 

Now we propose that every order $\alpha$ digital $(t,d)$-sequence over $\ZZ_2$ with $\alpha \ge 5$ that is constructed by interlacing of order 1 digital sequence with matrices of the form \eqref{def:gNs} has the optimal order of periodic $L_2$-discrepancy. In a nutshell, we formulate the following more concrete version of our main result (Theorem~\ref{thm1}). 

\begin{theorem}\label{thm2}
Let $d \in \NN$, $d \ge 2$. Let $\cS_{\alpha d}$ be a  usual (i.e., order 1) digital $(t,\alpha d)$-sequence over $\ZZ_2$ obtained from  matrices that are constructed like in \eqref{def:gNs} with $\alpha \ge 5$. Then we have 
\begin{align*}
 L_{2,N}^{{\rm per}}(\mathscr{D}_{\alpha}^d(\cS_{\alpha d})) \lesssim_{d}  \frac{(\log N)^{(d-1)/2}}{N} \sqrt{S(N)}\ \ \ \ \mbox{ for all $N\ge 2$.}
\end{align*}
\end{theorem}

\section{The Proof of the Results}\label{sec:aux}

Obviously, Theorem~\ref{thm2} implies Theorem~\ref{thm1}. The proof of Theorem~\ref{thm2} will be based on a Walsh series representation of the squared periodic $L_2$-discrepancy.

\paragraph{Walsh functions.} For $k \in \NN_0$ the $k$-th Walsh function (in base 2) $\wal_{k}:[0,1) \rightarrow \{-1,1\}$ is defined in the following way: let $k$ have binary representation
\[
   k = \kappa_{a-1} 2^{a-1} + \cdots + \kappa_1 2 + \kappa_0,
\]
with $\kappa_i \in \{0,1\}$, and let $x \in [0,1)$ have binary expansion $$x =
\frac{\xi_1}{2}+\frac{\xi_2}{2^2}+\cdots $$ with $\xi_i \in \{0,1\}$ (unique in the sense that
infinitely many of the $\xi_i$ must be zero), then
\[
  \wal_{k}(x) := (-1)^{\xi_1 \kappa_0 + \cdots + \xi_a \kappa_{a-1}}.
\]
 
For dimension $d \geq 2$, vectors $\bsk = (k_1, \ldots, k_d) \in \NN_0^d$ and $\bsx =
(x_1,\ldots,x_d) \in [0,1)^d$ we write
\[
   \wal_{\bsk}(\bsx) := \prod_{j=1}^d\wal_{k_j}(x_j).
\]
A collection of useful properties of Walsh functions can be found in \cite[Appendix~A]{DP10}.

We call $x\in [0,1)$ a dyadic rational if it can be written in a finite binary expansion. By $\oplus$ we denote the digit-wise addition modulo $2$. In order to avoid ambiguities and since it suffices for our purpose we define $\oplus$ just for dyadic rationals. For dyadic rationals $x = \sum_{i=1}^{\infty}
\frac{\xi_i}{2^i}$ and $y = \sum_{i=1}^{\infty} \frac{\eta_i}{2^i}$ with $\xi_i,\eta_i \in \{0,1\}$, where both digit sequences  $(\xi_i)_{i \ge 1}$ and $(\eta_i)_{i \ge 1}$ become eventually zero, we put
$$ x \oplus y := \sum_{i=0}^{\infty } \frac{\zeta_i}{2^i}, \; \; \;{\rm where}
\; \; \; \zeta_i := \xi_i + \eta_i \mypmod{2}.$$  Obviously $x \oplus y$ is then also a dyadic rational. For vectors $\bsx, \bsy \in [0,1)^d$ whose components are dyadic rationals we set $\bsx \oplus \bsy := (x_1 \oplus y_1, \ldots, x_d \oplus y_d)$. 

For any $\bsk \in \NN_0^d$ and $\bsx,\bsy \in [0,1)^d$ whose components are dyadic rationals we have 
\begin{equation}\label{wal:mult}
\wal_{\bsk}(\bsx \oplus \bsy)=\wal_{\bsk}(\bsx)\wal_{\bsk}(\bsy)
\end{equation}

It can be shown (see \cite[Lemma~4.72]{DP10}) that any digital net $\cP_{2^m,d}$ (see Remark~\ref{re:dnet}) is a subgroup of $([0,1)^d,\oplus)$ and furthermore, all components of every point of $\cP_{2^m,d}$ are dyadic rationals. From \eqref{wal:mult} it follows that $\wal_{\bsk}$ is a character of the group $(\cP_{2^m,d},\oplus)$. Hence, for any digital net $\cP_{2^m,d}$ with generating matrices $C_1,\ldots,C_d \in \ZZ_2^{m \times m}$ and any $\bsk=(k_1,\ldots,k_d) \in \NN_0^d$ it follows that 
\begin{equation}\label{charprop}
\frac{1}{2^m}\sum_{h=0}^{2^m-1} \wal_{\bsk}(\bsx_h)=\left\{
\begin{array}{ll}
1 & \mbox{ if } C_1^{\top} \vec{k_1}+\cdots +C_d^{\top} \vec{k_d}=\vec{0},\\
0 & \mbox{ otherwise}. 
\end{array}\right.
\end{equation}
For a proof of this fact we refer to \cite[Lemma~ 4.75]{DP10}. This property is called the {\it character property of digital nets}.

\paragraph{Walsh series expansion of the periodic $L_2$-discrepancy.} The starting point is the representation of the periodic $L_2$-discrepancy in terms of exponential sums.  Re-writing  \eqref{pr_dia} we obtain 
\begin{equation*}
(L_{2,N}^{{\rm per}}(\cP))^2 =  -\frac{1}{3^d}  + \frac{1}{3^d N^2} \sum_{n,p=0}^{N-1}  K_d(\bsx_n,\bsx_p),
\end{equation*}
where $$K_d(\bsx,\bsy):=   \sum_{\bsh \in \ZZ^d} \frac{1}{r(\bsh)^2} \exp(2 \pi \icomp \bsh \cdot (\bsx-\bsy))=\prod_{j=1}^d \left( \sum_{h = -\infty}^{\infty} \frac{1}{r(h)^2} \exp(2 \pi \icomp h (x_j-y_j)\right).$$
Now we expand $K_d(\bsx,\bsy)$ into a Walsh series. We have $$K_d(\bsx,\bsy)=\sum_{\bsk,\bsell \in \NN_0^d} \rho(\bsk,\bsell) \wal_{\bsk}(\bsx) \wal_{\bsell}(\bsy),$$ where $$\rho(\bsk,\bsell)=\int_{[0,1]^{2d}} K_d(\bsx,\bsy) \wal_{\bsk}(\bsx) \wal_{\bsell}(\bsy) \rd \bsx \rd \bsy.$$

Due to the multiplicative structure of $K_d$ and of the multi-dimensional Walsh functions, for $\bsk=(k_1,\ldots,k_d)$ and $\bsell=(l_1,\ldots,l_d)$ in $\NN_0^d$ we have $$\rho(\bsk,\bsell)=\prod_{j=1}^d \rho(k_j,l_j),$$
where, for $d=1$ and $k,l \in \NN_0$, 
\begin{eqnarray*}
\rho(k,l) & = & \int_0^1 \int_0^1  K_1(x,y) \wal_{k_j}(x) \wal_{l_j}(y) \rd x \rd y\\
& = & \int_0^1 \int_0^1 \sum_{h=-\infty}^{\infty} \frac{1}{r(h)^2} \exp(2\pi \icomp h(x-y)) \wal_k(x) \wal_{l}(y) \rd x \rd y\\
& = & \sum_{h=-\infty}^{\infty} \frac{1}{r(h)^2}  \left(\int_0^1 \exp(2 \pi \icomp hx) \wal_k(x)  \rd x \right)\left(\int_0^1 \exp(-2 \pi \icomp h y) \wal_l(y)  \rd y \right)\\
& = & \sum_{h=-\infty}^{\infty} \frac{1}{r(h)^2}  \left(\int_0^1 \exp(2 \pi \icomp hx) \wal_k(x)  \rd x \right)\overline{\left(\int_0^1 \exp(2 \pi \icomp h y) \wal_l(y)  \rd y \right)}.
\end{eqnarray*}  
Put $$\beta_{h,k}:=\int_0^1 \exp(2 \pi \icomp hx) \wal_k(x)  \rd x.$$ Note that $\beta_{0,0}=1$, $\beta_{0,k}=0$ for $k \in \NN$ and $\beta_{h,0}=0$ for $h \in \ZZ\setminus\{0\}$. Then we can write $$\rho(k,l) =  \sum_{h=-\infty}^{\infty} \frac{\beta_{h,k} \overline{\beta}_{h,l}}{r(h)^2}$$ and, in particular,  $\rho(0,0)=1$ and $\rho(0,l)=\rho(k,0)=0$ for $k,l \in \NN$. For $k,l \in \NN$ we have $$\rho(k,l) =  \sum_{h=-\infty\atop h \not=0}^{\infty} \frac{\beta_{h,k} \overline{\beta}_{h,l}}{r(h)^2}=\frac{6}{4 \pi^2} \sum_{h=-\infty\atop h \not=0}^{\infty} \frac{\beta_{h,k} \overline{\beta}_{h,l}}{h^2}=6 \gamma_2(k,l),$$ where $\gamma_2(k,l)$ are defined in \cite[Eq.~(14.17) on p.~453]{DP10} as the Walsh coefficients of the normalized and periodically extended second Bernoulli polynomial. These coefficients are computed in \cite[Lemma~14.17]{DP10} and from there we obtain:

\begin{lemma}\label{le:rho}
We have $\rho(0,0)=1$ and $\rho(k,0) = \rho(0,l) = 0$ for all $k,l \in \NN$. For $k \in \NN$ write 
$k=2^{a_1-1}+2^{a_2-1}+\cdots+2^{a_v-1}$ with $v \in \NN$ and $a_1>a_2> \dots > a_v \ge 1$, and $k'=k-2^{a_1-1}$ and $k''=k-2^{a_1-1}-2^{a_2-1}=k'-2^{a_2-1}$. Likewise, for $l \in \NN$ write $l=2^{b_1-1}+2^{b_2-1}+\cdots + 2^{b_w-1}$ with $w \in \NN$ and $b_1 > b_2 > \dots > b_w \ge 1$, and $l'=l-2^{b_1-1}$ and $l''=l-2^{b_1-1}-2^{b_2-1}=l'-2^{b_2-1}$. Then for all $k,l \in \NN$ we have
\begin{equation*}
\rho(k,l) = \left\{\begin{array}{ll}  2^{-2 a_1} & \mbox{if } k = l, \\
2^{-a_1-b_1-2} & \mbox{if } k' = l' > 0 \mbox{ and } k \neq l, \\  
-2^{-a_1-a_2-2} & \mbox{if } k'' =  l,  \\
-2^{-b_1-b_2-2}& \mbox{if } k = l'',  \\ 
0 & \mbox{otherwise}.  
\end{array}\right.
\end{equation*}
\end{lemma}

Using the Walsh series expansion of $K_d(\cdot,\cdot)$ we can write the periodic $L_2$-discrepancy  in the following way:

\begin{lemma}\label{lem_r}
For any point set $\cP_{N,d}=\{\bsx_0,\bsx_1,\ldots, \bsx_{N-1}\}$ in $[0,1)^d$ we have
\begin{eqnarray}\label{fo:perL2wal}
(L_{2,N}^{{\rm per}}(\cP_{N,d}))^2  = \frac{1}{3^d N^2} \sum_{n,p=0}^{N-1} \sum_{\bsk,\bsell \in \NN_0^d \setminus\{ \bszero\}} \rho(\bsk,\bsell) \wal_{\bsk}(\bsx_n) \wal_{\bsell}(\bsx_p). 
\end{eqnarray}
where $\bsk = (k_1,\ldots, k_d)$, $\bsell = (l_1,\ldots, l_d)$,
$\rho(\bsk,\bsell) := \prod_{j=1}^d \rho(k_j,l_j)$, where, for $k, l \in \NN_0$ the coefficients $\rho(k,l)$ are given in Lemma~\ref{le:rho}. 
\end{lemma}

For digital nets we can simplify the above formula further. But first introduce the notion of a digitally shifted digital net. Let $\cP_{2^m,d}=\{\bsx_0,\bsx_1,\ldots, \bsx_{2^m-1}\}$ in $[0,1)^d$ be a digital net over $\ZZ_2$ and let $\bssigma \in [0,1)^d$ (in our context we restrict to $\bssigma$ with dyadic rationals as components). Then we call the point set $\cP_{2^m,d}(\bssigma)=\{\bsx_0 \oplus \bssigma,\bsx_1 \oplus \bssigma, \ldots, \bsx_{2^m-1} \oplus \bssigma\}$ a {\it digitally shifted digital net over $\ZZ_2$}.

\begin{lemma}\label{lem_l2formula}
If $\cP_{2^m,d}$ is a digital net over $\ZZ_2$, then the squared periodic $L_2$-discrepancy of $\cP_{2^m,d}$ is given by 
\begin{equation*}
(L_{2,2^m}^{{\rm per}}(\cP_{2^m,d}))^2 =  \frac{1}{3^d} \sum_{\bsk,\bsell \in \Dcal^\ast} \rho(\bsk,\bsell),
\end{equation*}
where $\Dcal^\ast=\Dcal\setminus \{\bszero\}$ and where $\Dcal$ is the so-called dual net given by $$\Dcal=\{(k_1,\ldots,k_d) \in \NN_0^d\, : \, C_1^{\top}\vec{k}_1+\cdots+C_d^{\top}\vec{k}_d=\vec{0}\},$$ where for $k \in \NN_0$ with binary expansion $k=\kappa_0+\kappa_1 2+\kappa_2 2^2+\cdots$ we put $\vec{k}=(\kappa_0,\ldots,\kappa_{m-1})^{\top}$. Again, the coefficients $\rho(\bsk,\bsell)$ are given as in Lemma~\ref{lem_r}.

If $\cP_{2^m,d}(\bssigma)$ is a  digitally shifted digital net over $\ZZ_2$, then we have
\begin{equation*} 
(L_{2,2^m}^{{\rm per}}(\cP_{2^m,d}(\bssigma)))^2 = \frac{1}{3^d} \sum_{\bsk,\bsell \in \Dcal^\ast} \rho(\bsk,\bsell) \wal_{\bsk}(\bssigma) \wal_{\bsell}(\bssigma).
\end{equation*}
\end{lemma}

\begin{proof}
From Lemma~\ref{lem_r} we obtain
\begin{eqnarray*}
(L_{2,2^m}^{{\rm per}}(\cP_{2^m,d}))^2  & = & \frac{1}{3^d 2^{2m}} \sum_{n,p=0}^{2^m-1} \sum_{\bsk,\bsell \in \NN_0^d \setminus\{ \bszero\}} \rho(\bsk,\bsell) \wal_{\bsk}(\bsx_n) \wal_{\bsell}(\bsx_p)\\
& = & \frac{1}{3^d} \sum_{\bsk,\bsell \in \NN_0^d \setminus\{ \bszero\}} \rho(\bsk,\bsell) \left(\frac{1}{2^m}\sum_{n=0}^{2^m-1}\wal_{\bsk}(\bsx_n)\right)\left(\frac{1}{2^m}\sum_{p=0}^{2^m-1} \wal_{\bsell}(\bsx_p)\right).
\end{eqnarray*}
Now the first result follows from \eqref{charprop}.

The second part follows in the same manner using \eqref{wal:mult}.
\end{proof}

With Lemma~\ref{lem_r} and Lemma~\ref{lem_l2formula} we are in a comparable situation to \cite{DP14}. Instead of  \cite[Lemma~2.2 and 2.4]{DP14} for the proof of \cite[Theorem~1.1]{DP14} now Lemma~\ref{lem_r} and Lemma~\ref{lem_l2formula} will serve as starting point of the proof of the present Theorem~\ref{thm2}. The main difference is in the definition of the coefficients $\rho(\bsk,\bsl)$ here and $r(\bsk,\bsl)$ in \cite{DP14}. However, we will see that $\rho(\bsk,\bsl)$ can be estimated in the same way as $r(\bsk,\bsl)$ in \cite{DP14}. The crucial part of the proof is to estimate the periodic $L_2$-discrepancy by an expression, that appears as an upper bound on the standard $L_2$-discrepancy in \cite[Lemma~3.1]{DP14}. Then one can use exactly the same procedure like in \cite{DP14} in order to finish the proof of our Theorem~\ref{thm2}. We will show in detail how to derive the mentioned upper bound on the periodic $L_2$-discrepancy.

\begin{proof}[Proof of Theorem~\ref{thm2}]
Let $\cS_{\alpha d}$ be a digital sequence in dimension $\alpha d$ based on the construction \eqref{def:gNs} and apply the digit interlacing function $\mathscr{D}_\alpha^d$ of order $\alpha$. The resulting sequence $\cS_d:=\mathscr{D}_\alpha^d(\cS_{\alpha d}):=(\bsx_0, \bsx_1,\bsx_2,\ldots)$ in $[0,1)^d$ is an order $\alpha$ digital $(t,d)$-sequence with $t = \alpha \sum_{j=1}^{\alpha d} (e_j-1) + d {\alpha \choose 2}$. Using \eqref{eq_t_tprime}, $\cS_d$ is also an order $\alpha'$ digital $(t',d)$-sequence with $t' = \lceil t \alpha'/\alpha \rceil \le t$ for all $1 \le \alpha' \le \alpha$. Thus it is also an order $\alpha'$ digital $(t,d)$-sequence for all $1 \le \alpha' \le \alpha$. 

Let $C_1, \ldots, C_d$ denote the generating matrices of the digital sequence $\cS_d$. Let $C_{j, \mathbb{N} \times m}$ denote the first $m$ columns of $C_j$. Due to \eqref{matE:null} only the first $\alpha m$ rows of $C_{j, \mathbb{N} \times m}$ can be nonzero and hence $C_j$ is of the form
$$
C_{j} = \left( \begin{array}{ccc}
           & \vline &  \\
  C_{j,\alpha m \times m} & \vline & D_{j,\alpha m \times \NN} \\
           & \vline &   \\ \hline
    & \vline & \\
 0_{\NN \times m}  & \vline &  F_{j,\NN \times \NN} \\
   &   \vline       &
\end{array} \right) \in \ZZ_2^{\NN \times \NN},
$$  where $0_{\NN \times m}$ denotes the $\NN \times m$ zero matrix. Note that the entries of each column of the   matrix $F_{j,\NN \times \NN}$ become eventually zero.

We use Lemma~\ref{lem_r} to obtain
\begin{align*}
(L^{{\rm per}}_{2,N}(\cS_d))^2  = & \frac{1}{3^d N^2} \sum_{\bsk,\bsell \in \NN_0^d \setminus\{ \bszero\}} \rho(\bsk,\bsell) \left(\sum_{n=0}^{N-1}\wal_{\bsk}(\bsx_n)\right)\left(\sum_{p=0}^{N-1} \wal_{\bsell}(\bsx_p)\right).
\end{align*}

Let $N = 2^{m_1} + 2^{m_2} + \cdots + 2^{m_r}$ with $m_1 > m_2 > \cdots > m_r \ge 0$ (hence $r=S(N)$). We split the initial $N$ terms of the sequence $\cS_d$ into the point sets
\begin{equation*}
\cP_i:=\{\bsx_{2^{m_1}+\cdots + 2^{m_{i-1}}},\ldots, \bsx_{-1+2^{m_1} + \cdots + 2^{m_i}}\},
\end{equation*}
for $i \in \{1,\ldots,r\}$, where for $i=1$ we define $2^{m_1} + \cdots + 2^{m_{i-1}} = 0$. Any $n \in \{2^{m_1}+\cdots + 2^{m_{i-1}}, \ldots ,-1+2^{m_1} + \cdots + 2^{m_i}\}$ can be written in the form $$n=2 ^{m_1}+\cdots+2^{m_{i-1}}+a=2^{m_{i-1}} l +a$$ with $a \in \{0,1,\ldots,2^{m_i}-1\}$ and $l=1+2^{m_{i}-m_{i-1}}+\cdots+2^{m_1-m_{i-1}}$ if $i > 1$ and $l=0$ for $i=1$. Hence the binary digit vector of $n$ is given by $$\vec{n}=(a_0,a_1,\ldots,a_{m_i-1},l_0,l_1,l_2,\ldots)^{\top}=:{\vec{a} \choose \vec{l}},$$ where $a_0,\ldots,a_{m_i-1}$ are the binary digits of $a$ and $l_0,l_1,l_2,\ldots$ are the binary digits of $l$. With this notation we have
$$C_j \vec{n}=\left( \begin{array}{c} C_{j,\alpha m_i \times m_i}  \vec{a} \\ 0 \\ 0 \\ \vdots \end{array} \right)  + \left( \begin{array}{c}
             \\
  D_{j,\alpha m \times \NN}  \\
          \\ \hline  \\
 F_{j,\NN \times \NN}  \\
   \end{array} \right) \vec{l}.$$ For the point set $\cP_i$ under consideration, the vector
\begin{equation}\label{dig_shift}
\vec{\sigma}_{i,j}:=\left( \begin{array}{c}
             \\
  D_{j,\alpha m \times \NN}  \\
          \\ \hline  \\
 F_{j,\NN \times \NN}  \\
   \end{array} \right) \vec{l}
\end{equation}
is constant and its components become eventually zero (i.e., only a finite number of components is nonzero). Furthermore, $C_{j,\alpha m_i \times m_i}  \vec{a}$ for $a \in \{0,1,\ldots,2^{m_i}-1\}$ and $j \in \{1,\ldots,d\}$ generate an order $\alpha$ digital $(t,m_i,d)$-net over $\ZZ_2$ (which is also an order $\alpha'$ digital $(t,m_i,d)$-net over $\ZZ_2$ for $1 \le \alpha' \le \alpha$).

This means that the point set $\cP_i$ is a digitally shifted order $\alpha$ digital $(t,m_i,d)$-net over $\ZZ_2$  and the generating matrices
\begin{equation}\label{genmatpi}
C_{1, \alpha m_i \times m_i},\ldots, C_{d, \alpha m_i\times m_i}
\end{equation}
of this digital net are the left upper $\alpha m_i \times m_i$ submatrices of the generating matrices $C_1,\ldots, C_d$ of the digital sequence. We denote the digital shift, which is given by \eqref{dig_shift}, by $\bssigma_i$. Note that all the coordinates of the digital shift are dyadic rationals since the components of $\vec{\sigma}_{i,j}$ become eventually zero.

Let $\mathcal{D}_i$ denote the dual net corresponding to the digital net with generating matrices \eqref{genmatpi}, i.e.,
\begin{equation*}
\mathcal{D}_i = \{\bsk=(k_1,\ldots,k_d) \in \NN_0^d\ : \ C_{1, \alpha m_i \times m_i}^\top \vec{k}_1 + \cdots + C_{d, \alpha m_i \times m_i}^\top \vec{k}_d = \vec{0}\},
\end{equation*}
where for $k \in \NN_0$ with binary expansion $k = \kappa_0 + \kappa_1 2 + \kappa_2 2^2+\cdots$ we set $\vec{k} = (\kappa_0, \kappa_1,\ldots, \kappa_{\alpha m_i-1})^\top$. Set $\Dcal_i^{\ast}=\Dcal_i \setminus\{\bszero\}$.  Then we have
\begin{align*}
\sum_{n=0}^{N-1} \wal_{\bsk}(\bsx_n) = & \sum_{i=1}^r  \sum_{n=2^{m_1}+\cdots + 2^{m_{i-1}}}^{-1+2^{m_1}+\cdots + 2^{m_i}} \wal_{\bsk}(\bsx_n),
\end{align*}
where again for $i=1$ we set $2^{m_1}+\cdots + 2^{m_{i-1}} = 0$, and by the character property \eqref{charprop}
\begin{align*}
\sum_{n=2^{m_1}+\cdots + 2^{m_{i-1}}}^{-1+2^{m_1}+\cdots + 2^{m_i}} \wal_{\bsk}(\bsx_n) = & \left\{\begin{array}{ll} \wal_{\bsk}(\bssigma_i) 2^{m_i} & \mbox{if } \bsk \in \mathcal{D}_i, \\ 0 & \mbox{if } \bsk \notin \mathcal{D}_i. \end{array} \right.
\end{align*}

Therefore
\begin{align*}
(L^{{\rm per}}_{2,N}(\cS_d))^2  =  & \frac{1}{3^d N^2} \sum_{i, i' =1}^r 2^{m_i +m_{i'}} \sum_{\bsk \in \mathcal{D}^{\ast}_i, \bsell \in \mathcal{D}^{\ast}_{i'}} \rho(\bsk, \bsell) \wal_{\bsk}(\bssigma_i) \wal_{\bsell}(\bssigma_{i'}) \\ 
\le & \frac{1}{3^d N^2} \sum_{i, i' =1}^r 2^{m_i + m_{i'}} \sum_{\bsk \in \mathcal{D}^{\ast}_i, \bsell \in \mathcal{D}^{\ast}_{i'}} |\rho(\bsk, \bsell)|.
\end{align*}
Now define 
\begin{equation*}
\Ecal_{i,i'} = \{(\bsk,\bsell) \in \Dcal^\ast_i \times \Dcal^\ast_{i'} \ : \ \rho(\bsk,\bsell) \neq 0\}
\end{equation*}
and
\begin{equation*}
\Ecal_{i,i'}(z) = \{(\bsk,\bsell) \in \Ecal_{i,i'} \ : \ \mu(\bsk) + \mu(\bsell) = z\},
\end{equation*}
where the function $\mu: \mathbb{N}_0 \to \mathbb{N}_0$ is defined by $\mu(0):=0$ and for $k = \kappa_0 + \kappa_1 2 + \cdots + \kappa_{a-2} 2^{a-2} + 2^{a-1}$ with $\kappa_j \in\{0,1\}$ by $\mu(k):=a$ and, for $\bsk=(k_1,\ldots,k_d) \in \NN_0^d$, we put $\mu(\bsk):=\mu(k_1)+\cdots+\mu(k_d)$.

From Lemma~\ref{le:rho} it follows that for $(\bsk,\bsell) \in \Ecal_{i,i'}$ we have $$|\rho(\bsk,\bsell)| \le \frac{1}{2^{\mu(\bsk) + \mu(\bsell)}}.$$  Thus we have 
\begin{equation}\label{bd:l2per1}
(L^{{\rm per}}_{2,N}(\cS_d))^2  \le \frac{1}{3^d N^2}  \sum_{i,i'=1}^r 2^{m_i +m_{i'}} \sum_{(\bsk,\bsell) \in \Ecal_{i,i'}} \frac{1}{2^{\mu(\bsk) + \mu(\bsl)}}.
\end{equation}

Now we re-order the sum over all $(\boldsymbol{k},\boldsymbol{l}) \in \Ecal_{i,i'}$ according to the value of $\mu(\boldsymbol{k}) + \mu(\boldsymbol{l})$. Assume that $\bsk=(k_1,\ldots,k_d) \in \Dcal_i^\ast$. Let $k_j=\kappa_{j,0}+\kappa_{j,1} 2 +\cdots +\kappa_{j,a_j-2} 2^{a_j-2}+2^{a_j -1}$ with $a_j=\mu(k_j)$ for $j \in \{1,\ldots,d\}$. Let further $\vec{c}_{j,u}$ denote the $u$-th row vector of the matrix $C_{j,\alpha m_i \times m_i}$. Then $$C_{1, \alpha m_i \times m_i}^\top \vec{k}_1 + \cdots + C_{d, \alpha m_i \times m_i}^\top \vec{k}_d = \vec{0}$$ is equivalent to $$\sum_{j=1}^d \left(\sum_{u=0}^{a_j-2} \vec{c}_{j,u+1}^{\ \top} \kappa_{j,u} + \vec{c}_{j,a_j-1}^{\ \top}\right)=\vec{0}.$$ Hence it follows from the linear independence property for the row vectors of generating matrices of digital nets in Definition~\ref{def_net} that $$\mu(\bsk)=a_1+\cdots+a_d > m_i -t.$$ In the same way $\bsl \in \Dcal_{i'}^\ast$ implies that $\mu(\boldsymbol{l}) > m_{i'}-t$. Hence $(\bsk,\bsl)\in \Dcal^\ast_i \times \Dcal^\ast_{i'}$ implies $\mu(\bsk) + \mu(\bsl) \ge m_i + m_{i'} -2t+2$.

Thus for the innermost sum in \eqref{bd:l2per1} we have $$\sum_{(\bsk,\bsl) \in \Ecal_{i,i'}} \frac{1}{2^{\mu(\bsk)+\mu(\bsl)}} = \sum_{z= m_i + m_{i'} - 2t + 2}^\infty \frac{|\Ecal_{i,i'}(z)|}{2^z}$$ and this implies

\begin{equation}\label{bd:l2per2}
(L^{{\rm per}}_{2,N}(\cS_d))^2  \le \frac{1}{3^d}  \sum_{i,i'=1}^r \frac{2^{m_i}}{N} \frac{2^{m_{i'}}}{N}  \sum_{z= m_i + m_{i'} - 2t + 2}^\infty \frac{|\Ecal_{i,i'}(z)|}{2^z}.
\end{equation}

Now the right hand side of the estimate \eqref{bd:l2per2} for the squared periodic $L_2$-discrepancy is -- up to the factor $1/3^d$ -- the same as the upper bound on the standard $L_2$-discrepancy in \cite[Lemma~3.1, Eq. (16)]{DP14}. Note that the sets $\Ecal_{i,i'}(z)$ are subsets of the corresponding sets $\Jcal_{i,i'}(z)$ used in \cite[Lemma~3.1, Eq. (15)]{DP14}. This means that from this stage on we can proceed in exactly the same way as in \cite{DP14} in order to obtain the proposed upper bound in Theorem~\ref{thm2}.
\end{proof}

\begin{remark}\rm
In \cite{DHMP} it is shown by a different proof method that even order $\alpha=2$ digital sequences suffice in order to obtain the optimal order of magnitude for the standard $L_2$-discrepancy. We conjecture that the same holds true for the periodic $L_2$-discrepancy.  
\end{remark}

\noindent{\bf Author's Address:}

\noindent Friedrich Pillichshammer, Institut f\"{u}r Finanzmathematik und Angewandte Zahlentheorie, Johannes Kepler Universit\"{a}t Linz, Altenbergerstra{\ss}e 69, A-4040 Linz, Austria. Email: friedrich.pillichshammer@jku.at\\


\begin{thebibliography}{99}
%\bibitem{BaKa} N. L. Bassily and I. K\'{a}tai, Distribution of the values of $q$-additive functions on polynomial sequences. Acta Math. Hungar., {\bf 68}, 353--361, 1995.

\bibitem{BC} J. Beck and W. W. L. Chen: {\it Irregularities of Distribution.} Cambridge University Press, Cambridge, 1987.

\bibitem{chafa} H. Chaix and H. Faure: Discr\'{e}pance et diaphonie en dimension un. Acta Arith. 63: 103--141, 1993.

%\bibitem{CS02} W. W. L. Chen and M. M. Skriganov, Explicit constructions in the classical mean squares problem in irregularity of point distribution.  J. Reine Angew. Math., {\bf 545}, 67--95, 2002.

%\bibitem{CS3}  W. W. L. Chen and M. M. Skriganov, Orthogonality and digit shifts in the classical mean squares problem in irregularities of point distribution.  In: {\it Diophantine approximation}, pp. 141--159, Dev. Math., 16, SpringerWienNewYork, Vienna, 2008.

%\bibitem{chrest} H.E. Chrestenson, A class of generalized Walsh functions, Pacific J. Math., {\bf 5}, 17--31, 1955.



%\bibitem{dav} H. Davenport, Note on irregularities of distribution. Mathematika, {\bf 3}, 131--135, 1956.

\bibitem{D07} J. Dick: Explicit constructions of quasi-Monte Carlo rules for the numerical integration of high-dimensional periodic functions. SIAM J. Numer. Anal. 45: 2141--2176, 2007.

\bibitem{D08} J. Dick: Walsh spaces containing smooth functions and quasi-Monte Carlo rules of arbitrary high order. SIAM J. Numer. Anal. 46: 1519--1553, 2008.

%\bibitem{D10} J. Dick, On quasi-Monte Carlo rules achieving higher order convergence. Monte Carlo and quasi-Monte Carlo methods 2008, 73--96, Springer, Berlin, 2009.

\bibitem{DHMP} J. Dick, A. Hinrichs, L. Markhasin, and F. Pillichshammer: Optimal $L_p$-discrepancy bounds for second order digital sequences. Israel J. Math. 221: 489--510, 2017.

\bibitem{DHP20} J. Dick, A. Hinrichs, and F. Pillichshammer: A note on the periodic $L_2$ discrepancy of Korobov's $p$-sets. Arch. Math. 115: 67--78, 2020.

%\bibitem{DP05} J. Dick and F. Pillichshammer, Multivariate integration in weighted Hilbert spaces based on Walsh functions and weighted Sobolev spaces, J. Complexity, {\bf 21},  149--195, 2005.

%\bibitem{DP05b} J. Dick and F. Pillichshammer: On the mean square weighted $L_2$ discrepancy of randomized digital $(t,m,s)$-nets over $\mathbb{Z}_2$. Acta Arith. 117: 371--403, 2005.

\bibitem{DP10} J. Dick and F. Pillichshammer: {\it Digital Nets and Sequences. Discrepancy Theory and Quasi-Monte Carlo Integration.} Cambridge University Press, Cambridge, 2010.

\bibitem{DP14} J. Dick and F. Pillichshammer: Optimal $\mathcal{L}_2$ discrepancy bounds for higher order digital sequences over the finite field $\mathbb{F}_2$. Acta Arith. 16(1): 65--99, 2014.

\bibitem{DT97} M. Drmota and R.F. Tichy: {\it Sequences, Discrepancies and Applications}. Lecture Notes in Mathematics 1651, Springer Verlag, Berlin, 1997.

\bibitem{F2005} H. Faure: Discrepancy and diaphony of digital (0,1)-sequences in prime base. Acta Arith. 117(2): 125--148, 2005. 

\bibitem{FKP} H. Faure, P. Kritzer, and F. Pillichshammer: From van der Corput to modern constructions of sequences for quasi-Monte Carlo rules. Indag. Math. 26: 760--822, 2015. 

%\bibitem{FauPi09a} H. Faure and F. Pillichshammer, $L_2$ discrepancy of two-dimensional digitally shifted Hammersley point sets in base $b$. In: {\it Monte Carlo and Quasi-Monte Carlo Methods 2008},   L'Ecuyer, P. and Owen, A. (eds.), pp. 355--368, Springer, Berlin, 2009.

%\bibitem{FauPi09} H. Faure and F. Pillichshammer, $L_p$ discrepancy of generalized two-dimensional Hammersley point sets. Monatsh. Math., {\bf 158}, 31--61, 2009.

%\bibitem{FauPiPriSch09} H. Faure, F. Pillichshammer, G. Pirsic, and W. Ch. Schmid, $L_2$ discrepancy of generalized two-dimensional Hammersley point sets scrambled with arbitrary permutations. Acta Arith., {\bf 141}, 395--418, 2010.

%\bibitem{Frolov} K. K. Frolov, Upper bound of the discrepancy in metric $L\sb{p}$, $2\leq p<\infty $, Dokl. Akad. Nauk SSSR, {\bf 252}, 805--807, 1980.

\bibitem{g96} V.S. Grozdanov: On the diaphony of one class of one-dimensional sequences. Int. J. Math. Math. Sci. 19: 115--124, 1996.

\bibitem{g99} V.S. Grozdanov: On the diaphony and star-diaphony of the semisymmetrical net of Roth. C. R. Acad. Bulgare Sci. 52(9-10): 19--22, 1999.

\bibitem{HKP20} A. Hinrichs, R. Kritzinger, and F. Pillichshammer: Extreme and periodic $L_2$ discrepancy of plane point sets.  Acta Arith. 199(2): 163--198, 2021. 

\bibitem{HMOU} A. Hinrichs, L. Markhasin, J. Oettershagen, and T. Ullrich: Optimal quasi-Monte Carlo rules on order 2 digital nets for numerical integration of multivariate periodic functions. Numer. Math. 134: 163--196, 2016.

\bibitem{HOe} A. Hinrichs and J. Oettershagen: Optimal point sets for quasi-Monte Carlo integration of bivariate periodic functions with bounded mixed derivatives. {\it Monte Carlo and Quasi-Monte Carlo Methods}, pp. 385--405, Springer Proc. Math. Stat., 163, Springer, [Cham], 2016. 

\bibitem{HiWe} A. Hinrichs and H. Weyhausen: Asymptotic behavior of average $L_p$-discrepancies. J. Complexity 28(4): 425--439, 2012.

\bibitem{HK} W. Hornfeck and Ph. Kuhn: Diaphony, a measure of uniform distribution, and the Patterson function. Acta Crystallogr. Sect. A 71(4): 382--391, 2015. 

\bibitem{kirk} N. Kirk: On Proinov’s lower bound for the diaphony. Uniform Distribution Theory 15(2): 39--72, 2020.

%\bibitem{KriPi2006} P. Kritzer and F. Pillichshammer, An exact formula for the $L\sb 2$ discrepancy of the shifted Hammersley point set. Unif. Distrib. Theory, {\bf 1}, 1--13, 2006.

\bibitem{KP22a} R. Kritzinger and F. Pillichshammer: Exact order of extreme $L_p$ discrepancy of infinite sequences in
arbitrary dimension. Arch. Math.  118: 169--179, 2022.

\bibitem{KP22b} R. Kritzinger and F. Pillichshammer: Point sets with optimal order of extreme and periodic discrepancy. Acta Arith. 204(3): 191--223, 2022.

\bibitem{kuinie} L. Kuipers  and H. Niederreiter: {\it Uniform Distribution of Sequences}. John Wiley, New York, 1974.

%\bibitem{lp} G. Larcher  and F. Pillichshammer, Walsh series analysis of the $\LL_2$-discrepancy of symmetrisized point sets. Monatsh. Math., {\bf 132}, 1--18, 2001.

%\bibitem{LP05} G. Larcher and F. Pillichshammer: Moments of the weighted sum-of-digits function.  Quaest. Math. 28: 321--336, 2005.


\bibitem{Lev}  V.F. Lev: On two versions of $L^2$-discrepancy and geometrical interpretation of diaphony. Acta Math. Hungar. 69(4): 281--300, 1995.

\bibitem{Lev0}  V.F. Lev: The exact order of generalized diaphony and multidimensional numerical integration. J. Austral. Math. Soc. Ser. A 66(1): 1--17, 1999.

%\bibitem{Man}  E. Manstavi\v{c}ius, Probabilistic theory of additive functions related to systems of numeration. In: {\it New trends in probability and statistics}, Vol. 4 (Palanga, 1996),  413--429, VSP, Utrecht, 1997.

\bibitem{mat} J. Matou\v{s}ek: {\it Geometric Discrepancy. An Illustrated Guide.} Algorithms and Combinatorics, 18. Springer-Verlag, Berlin, 1999.


%\bibitem{Nie73} H. Niederreiter, Application of Diophantine approximations to numerical integration. In: {\it Diophantine approximation and its applications (Proc. Conf., Washington, D.C., 1972)},  pp. 129--199. Academic Press, New York, 1973.

%\bibitem{nie86} H. Niederreiter, Low-discrepancy point sets, Monatsh. Math., {\bf 102}, 155--167, 1986.

\bibitem{nie87} H. Niederreiter: Point sets and sequences with small discrepancy. Monatsh. Math. 104: 273--337, 1987. 

\bibitem{nie92} H. Niederreiter: {\it Random Number Generation and Quasi-Monte Carlo Methods}. No. 63 in CBMS-NSF Series in Applied Mathematics, SIAM, Philadelphia, 1992.

\bibitem{NX96} H. Niederreiter and C. P. Xing: Low-discrepancy sequences and global function fields with many rational places. Finite Fields Appl. 2: 241--273, 1996.

\bibitem{pag} G. Pag\'{e}s: Van der Corput sequences, Kakutani transforms and one-dimensional numerical integration. J. Comput. Appl. Math. 44(1): 21--39, 1992. 

\bibitem{PS} F. Pausinger and W. Ch. Schmid: A good permutation for one-dimensional diaphony. Monte Carlo Methods Appl. 16(3-4): 307-–322, 2010. 

%\bibitem{pro85} P.D. Proinov, On the $L^2$ discrepancy of some infinite sequences. Serdica, {\bf 11}, 3--12, 1985.

\bibitem{pro2000} P.D. Proinov: {\it Quantitative Theory of Uniform Distribution and Integral Approximation.} University of Plovdiv, Bulgaria (2000). (In Bulgarian)

\bibitem{pro1988a} P.D. Proinov: Symmetrization of the van der {C}orput generalized sequences, Proc. Japan Acad. Ser. A Math. Sci. 64: 159--162, 1988.

\bibitem{pg} P.D. Proinov and V.S. Grozdanov: On the diaphony of the van der Corput-Halton sequence. J. Number Theory 30, 94--104, 1988.

\bibitem{Roth} K.F. Roth: On irregularities of distribution. Mathematika 1: 73--79, 1954.

%\bibitem{roth2} K. F. Roth, On irregularities of distribution III, Acta Arith., {\bf 35}, 373--384, 1979.

%\bibitem{Roth4} K. F. Roth, On irregularities of distribution IV. Acta Arith., {\bf 37}, 67--75, 1980.

%\bibitem{Skr} M. M. Skriganov, Harmonic analysis on totally disconnected groups and irregularities of point distributions. J. Reine Angew. Math., {\bf 600}, 25--49, 2006.

%\bibitem{SloWo} I. H. Sloan and H. Wo\'zniakowski, When are quasi-Monte Carlo algorithms efficient for high dimensional integrals? J. Complexity, \textbf{14}, 1--33, 1998.

\bibitem{sob67} I.M. Sobol: The distribution of points in a cube and the approximate evaluation of integrals. Zh. Vychisl. Mat. i Mat. Fiz. 7: 784--802, 1967.

%\bibitem{walsh} J.L. Walsh: A closed set of normal orthogonal functions. Amer. J. Math. 55: 5--24, 1923.

%\bibitem{war} T. T. Warnock, Computational investigations of low discrepancy point sets. In: {\it Applications of number theory to numerical analysis}, pp. 319--343. Academic Press, New York 1972.

\bibitem{weyl} H. Weyl: \"Uber die Gleichverteilung von Zahlen mod. Eins. (German) Math. Ann. 77(3): 313--352, 1916. 

%\bibitem{Woz} H. Wo\'{z}niakowski, Average case complexity of multivariate integration. Bull. Amer. Math. Soc. New Series, {\bf 24}, 185--194, 1991.

%\bibitem{SKZ} S. K. Zaremba, Some applications of multidimensional integration by parts. Ann. Polon. Math., {\bf 21}, 85--96, 1968.

\bibitem{zint} P. Zinterhof: \"Uber einige Absch\"atzungen bei der Approximation von Funktionen mit Gleichverteilungsmethoden. (German) \"Osterr. Akad. Wiss. Math.-Naturwiss. Kl. S.-B. II 185: 121--132, 1976.

\end{thebibliography}
\end{document}